\documentclass[11pt]{amsart}
\usepackage{pstricks,pst-plot}

\definecolor{Black}{cmyk}{0,0,0,1}
\definecolor{OrangeRed}{cmyk}{0,0.6,1,0}            
\definecolor{DarkBlue}{cmyk}{1,1,0,0.20}
\definecolor{myblue}{rgb}{0.66,0.78,1.00}
\definecolor{Violet}{cmyk}{0.79,0.88,0,0}
\definecolor{Lavender}{cmyk}{0,0.48,0,0}

\parskip=\smallskipamount

\newtheorem{theorem}{Theorem}[section]
\newtheorem{lemma}[theorem]{Lemma}
\newtheorem{corollary}[theorem]{Corollary}

\theoremstyle{definition}
\newtheorem{definition}[theorem]{Definition}
\newtheorem{example}[theorem]{Example}

\newcommand{\bea}{\begin{eqnarray*}}
\newcommand{\eea}{\end{eqnarray*}}

\numberwithin{equation}{section}

\begin{document}

\title[ COMPLEX DYNAMICS]{ COMPLEX DYNAMICS WITH FOCUS ON THE REAL PART}
%
%
\subjclass[2000]{}
\date{\today}

\subjclass[2000]{}
\date{\today}

\author{John Erik Forn\ae ss}
\address{Department for Mathematical Sciences\\
Norwegian University of Science and Technology\\
Trondheim, Norway}
\email{john.fornass@math.ntnu.no}

\author{Han Peters}
\address{KdV Institute for Mathematics\\
University of Amsterdam\\
The Netherlands}
\email{h.peters@uva.nl}

\begin{abstract} We consider the dynamics of holomorphic polynomials in $\mathbb C$. We show that
the ergodic properties of the map can be seen already from the real parts of the orbits.
\end{abstract}

\maketitle

\tableofcontents

\section{Introduction}

Science is concerned with describing the world and mathematics is an important tool. One can use
mathematics to obtain equations for how a system will change over time. To get manageable equations,
the usual procedure is to suppress some of the parameters which play a role in the system.
This simplifies the mathematical equations, but it is important to ask whether the obtained results
still accurately describe the original situation. In this paper, we will investigate rigorously whether
one can recover precise results when one suppresses some variables. We will do this
in the case of one dimensional complex dynamics, where the exact theory is highly developed.

Questions similar to those addressed above were studied by Takens in \cite{Takens}, where the following result was proved.

\begin{theorem}[Takens]
Let $M$ be a compact manifold of dimension $m$. For pairs $(\phi, y)$, $\phi: M \rightarrow M$ a $\mathcal{C}^2$ diffeomorphism and $y: M \rightarrow \mathbb R$ a $\mathcal C^2$ function, it is a generic property that the map $\Phi_{(\phi, y)}: M \rightarrow \mathbb R^{2m+1}$, defined by
$$
\Phi_{(\phi, y)}(x) = (y(x), y(\phi(x)), \ldots , y(\phi^{2m}(x)))
$$
is an embedding.
\end{theorem}

Hence all information about the original dynamical system can be retrieved from the suppressed dynamical system. When the map $\phi$ is not injective one should not expect the Takens' Theorem to hold. For example, one could have distinct points $z,w \in X$ with $y(z) = y(w)$ and $\phi(z) = \phi(w)$. The points $z, w$ will be identified once the other variables are suppressed. In fact, such identifications will always occur in our complex analytic setting.

Let $P:\mathbb C \rightarrow \mathbb C$ be a polynomial of degree $d\geq 2.$ Let $z=x+iy$ denote coordinates in $\mathbb C.$ We use the notation $z_n=x_n+iy_n$ to denote
an orbit, $z_n=P^n(x_0+iy_0).$ We will consider the real orbits $\{x_n\}_{n\geq 0}.$

We prove that it will suffice to consider only the first $N(P)$ terms:

\begin{lemma}\label{lemma1.2}
Let $P(z)$ be a complex polynomial of degree $d \geq 2.$ Then there exists an integer $N=N(P)$
so that if $\{x_n+iy_n\}$ and $\{u_n+iv_n\}$ are two orbits and $x_n=u_n$ for $n\leq N$, then
$x_n=u_n$ for all $n.$
\end{lemma}

Let $\Phi:\mathbb C\rightarrow \mathbb R^{N+1}$, denote the map
$\Phi(z_0)=(x_0,\dots,x_N)$, and let $S=\Phi(\mathbb C).$ For $z_0\in \mathbb C$ we define the map
$Q:S \rightarrow S$ by $Q \circ \Phi = \Phi \circ P$, which gives
$$
Q(x_0,\dots, x_N)=(x_1,\dots,x_{N+1}).
$$

Even though our map $\Phi$ will not be an embedding, it is possible that many properties of the dynamical system $(\mathbb C, P)$ can still be observed for the system $(S, Q)$. We will focus on ergodic theoretic aspects of the dynamical systems. The following is a classical result by Brolin, Lyubich and Mane, see \cite{Brolin}, \cite{Lyubich}, and \cite{Mane}.

\begin{theorem}[Brolin, Lyubich, Mane]
There is a unique invariant, ergodic probability measure $\mu$ on $\mathbb C$ of maximal entropy, $\log d.$
\end{theorem}

We define a probability measure $\nu$ on $S$ by $\nu=\Phi_*(\mu)$. Our main result is the following:

\begin{theorem}\label{main}
Let $P$ be a non-exceptional complex polynomial of degree $d \geq 2$.
Then the probability measure $\nu$ is invariant and ergodic. Moreover it is the unique measure of maximal entropy, $\log d.$
\end{theorem}

The plan of the paper is the following. In the next section we introduce the concept of non-exceptional polynomials. We also prove Lemma \ref{lemma1.2} and provide some basic estimates on orbits.
In Section 3, we investigate mirrored orbits, i.e. points $z_0,w_0$ for which $x_n=u_n$ for all $n.$
Topological entropy on $S$ is introduced in Section 4 where it is proved that the entropy is $\log d$.
The metric entropy on $S$ is defined in Section 5, and in Section 6 we prove Theorem 1.4.

The second author was supported by a SP3-People Marie Curie Actionsgrant in the project Complex Dynamics (FP7-PEOPLE-2009-RG, 248443).



\section{Preliminary results}

\begin{lemma}
Let $P(z)$ denote a complex polynomial of degree $d\geq 2.$ Then there can be at most one vertical line which is mapped to a vertical line. For all other lines, the number of points in any given vertical line which is mapped to any other given vertical line is at most $d$.
\end{lemma}

\begin{proof}
For notational simplicity we work with horizontal lines instead.
If a horizontal line is mapped to a horizontal line, then after translations we can arrange that both lines are equal to the $x$ axis. Let us write $P(z)=\sum_{n=0}^d a_nz^n.$ The invariance of the real axis is equivalent to all $a_n$ being real. We show that no other horizontal line is mapped to a horizontal line.
We may assume that if there exists another horizontal line which is mapped to a horizontal line, then it is the line given by $y=1.$ Let $f(x)=\sum a_n(x+i)^n$ as a function of $x\in \mathbb R.$ Since the imaginary part is constant, the derivative must be real.
Hence $\sum_{n=1}^d na_n (x+i)^{n-1}$ is real valued.
But then $\sum_{n=1}^d na_n (x+i)^{n-1}-\sum_{n=1}^d na_n x^{n-1}\in \mathbb R.$
By looking at terms of order $d-2$ in $x$ we see that
$i (d(d-1)a_d x^{d-2})+ {\mathcal O}(x^{d-3})$ is purely real. This is not possible for large $x.$

It remains to be shown that if the image $A$ of a line is not included in another line $B$, then at most $d$ points in $A$ are mapped to $B.$ We may assume both lines are the $x$ axis, and the
map has the form $P=\sum_{n=0}^d a_n z^n.$ We write $a_n=b_n+ic_n$, so on the $x$ axis the map looks like
$$P(x)= \sum b_n x^n+ i \sum c_n x^n.$$
The imaginary part is a non-zero polynomial of degree at most $d$, so can have at most $d$ real zeros (counted with multiplicity).
\end{proof}

For polynomials of degree $3$, all polynomials with real coefficients must map some vertical line to a vertical line. We note that this is not the case for any other degree greater or equal to $2$.

\begin{lemma}
Let $P$ be a polynomial of degree $3$ with real coefficients. Then it must send some vertical line to some vertical line.
\end{lemma}

\begin{proof}
After real scaling and real translations in the domain and range, we can write $P(z)=z^3+az$ for some real $a$. But then the imaginary axis is mapped to itself.
\end{proof}

We prove Lemma 1.2.

\begin{lemma}
Let $P$ be a polynomial of degree $d \geq 2$. There exists an integer $N(P)$ so that for any two real orbits $\{x_n\}_{n\geq 0}$ and $\{u_n\}_{n \geq 0}$ with 
$x_n=u_n$ for $n= 0,1,...,N(P)$, we have $x_n=u_n$ for all $n\geq 0.$
\end{lemma}

\begin{proof}
Let $Q_n(x,y,u,v)=\mbox{Re} (P^n(x+iy)-P^n(u+iv))$, and set 
$$
Z_n=\{Q_0= \cdots = Q_n=0\}.
$$
Then for some
$N(P)$ we have that $Z_n=Z_{N(P)}$ for all $n \geq N(P).$
\end{proof}

We let $S,\Phi,Q$ be as in the introduction. Then $S\subset \mathbb R^{N+1} \subset
\mathbb R^{N+1}\cup \infty$ which is a subset of the $N+1$ sphere.
We extend $P$ to $\mathbb C \cup \{\infty\}=\overline{\mathbb C}$, i.e. the Riemann sphere by mapping the point at infinity to itself. We call the extension $\overline{P}.$

Note that we can conjugate with any affine map of the form $az+b$ with $a$ nonzero real, because
these maps preserve vertical lines. Hence we can assume the polynomial is of the form

$$
P(z)= e^{i\theta} z^d+ \mathcal O(z^{d-2}).
$$

\begin{definition}
If $P$ maps a vertical line to itself, we say that  $P$ is strongly exceptional.
\end{definition}

We give an example of a strongly exceptional map.

\begin{example}\label{ex:stronglyexceptional}
Let $P(z)=-iz^2+ia$ for a real number $a$. Then $P(0+iy)= i[y^2+a]$. Hence the imaginary axis is invariant and the dynamics on the $y$ axis is $y \rightarrow y^2+a$. For large $a$ the Julia set is contained in the $y$ axis. For this map the measure $\mu$ of maximal entropy is supported on the $y$ axis, so the push forward $\nu$ to $S$ is the Dirac mass at $0$. In particular the result in Theorem 1.4
does not hold.
\end{example}

The following is clear.

\begin{lemma}
A polynomial $P(z)=a_d z^d+ \sum_{j=0}^{d-2}a_j z^j$ is strongly exceptional if and
only if the $y$ axis is mapped to itself, which is equivalent to $a_ji^{j-1}$ being real for all $j \le d$.
\end{lemma}

\begin{definition}
We say that $P(z)=a_d z^d+ \sum_{j=0}^{d-2} a_jz^j$ is weakly exceptional if
$a_d i^{d-1}$ is real, but there is at least one $0\leq j\leq d-2$ for which $a_j i^{j-1}$
is not real.
\end{definition}

If $P$ is weakly exceptional, then the set $\overline{S}:=S \cup {\infty}$ is not compact. We will exclude this case from
consideration. In fact,

\begin{theorem}\label{thm:2.8}
If $P$ is exceptional, then  the image $\Phi(\mathbb C)$ is not closed in $\mathbb R^{N+1}.$
\end{theorem}
\begin{proof}
First suppose that $P$ is strongly exceptional. Without loss of generality we may assume that $P$ maps the imaginary axis to itself. Let $x \in \mathbb R \setminus \{0\}$ and consider points $z_k = x + i k$. Then for $k \in \mathbb N$ large enough we obtain a sequence of pre-images $w_1, \ldots, w_N$, with $P(w_{n+1}) = w_n$ and $P(w_1) = z_k$, for which $w_1, \ldots w_N$ converge to the imaginary axis as $k \rightarrow \infty$. Therefore the points $(\mathrm{Re}(w_N), \ldots, \mathrm{Re}(w_1), x)$ lie in $\Phi(\mathbb C)$ and converge to $(0, \ldots, 0 , x)$ as $k \rightarrow \infty$. But $(0, \ldots, 0 , x) \notin \Phi(\mathbb C)$, hence $\Phi(\mathbb C)$ is not closed in $\mathbb R^{N+1}$.

Now suppose that $P$ is weakly exceptional. We use horizontal lines instead of vertical lines for convenience.
Note that since the real axis is not mapped to itself, there can be  at most $d$ points on the real axis for which $y_0=y_1=0.$ Hence if we can show if there are more than
$d$ values for $t$ such that $(0,0,\dots,0,t)$ is in the closure of $S$, then we have shown that $S$ is not closed.
We can assume that $P(z)=z^d+a_{d-2}z^{d-2}+\cdots.$ Let $w$ denote B\"ottcher
coordinates near infinity. We then easily see that $w=z+\frac{\alpha}{z}+\mathcal O(\frac{1}{z^2})+\cdots$ or equivalently that $z=w-\frac{\alpha}{w}+\cdots.$ The real line is invariant in
the B\"ottcher coordinate, so there is an invariant curve for $P(z)$ of the form $y=u-\frac{\alpha}{u}+\cdots.$ 
Given an initial point $z_0$ corresponding to $w_0=re^{i\epsilon}$, we get $w_k=r^{d^k}e^{d^k i \epsilon}$ and then $z_k=r^{d^k}e^{d^k i \epsilon}-\frac{\alpha}{r^{d^k}e^{d^k i \epsilon}}+\cdots.$
We see then that $y_k=r^{d^k}d^k \epsilon+\cdots.$
We choose $\epsilon$ so that $r^{d^N}d^N \epsilon=t.$ The conclusion follows.
\end{proof}

We next study polynomials which are neither strongly nor weakly exceptional. We call these non-exceptional. They are characterized by the condition that $a_d i^d$ is not purely imaginary.

\begin{theorem}\label{thm:continuity}
If $P$ is a non-exceptional polynomial, then $S\cup\infty=\overline{S}$ is compact. Moreover, the map $\Phi$ extends to a continuous map from $\mathbb P^1$ to $\overline{S}$ by sending infinity to infinity. Similarly, $Q$ extends to a continuous map from $\overline{S}$ to itself by sending infinity to infinity.
\end{theorem}

To prove the theorem it is sufficient to show that if $\{z_n\}$ form an unbounded sequence in $\mathbb C$, then the images $\{\Phi(z_n)\}$ form an unbounded sequence in $\mathbb R^{N+1}.$ We do this by estimating orbits.

We write $a_di^d=A+iB$. Then the assumtion that $P$ is non-exceptional gives $A \neq 0$. Pick some small $\eta>0$ so that $\eta(1+\eta)^{d-1}\leq   \frac{|A|}{4 d|a_d|}.$ Note that by the mean value theorem $|(1+z)^d-1| \leq d (1+|z|)^{d-1}|z|$ for any complex number $z$.

\begin{lemma}\label{lemma2.10}
There exists an $R>0$ so that if $|x|\leq \eta |y|$ and $|y| \geq R$ then
$|x_1|\geq \frac {|A y^d|}{2}.$
\end{lemma}
\begin{proof}
$$
\begin{aligned}
P(z)-a_d (iy)^d & = a_d(iy)^d\left((1+(x/(iy))^d-1\right)+\sum_{n<d}a_n(x+iy)^n\\
|x_1-Ay^d| & \leq |a_d y^d| d(1+\eta)^{d-1} \eta+ \frac{|A y^d|}{4}\\
& \leq \frac{|Ay^d|}{2}
\end{aligned}
$$
\end{proof}

We set $\sigma:=\frac{\eta}{1+\eta}$, i.e. $\frac{1-\sigma}{\sigma}=\frac{1}{\eta}.$

\begin{lemma}
There exists an $R'>0$ so that if $|x|\geq \eta |y|$ and $|x|\geq R'$, then
\begin{enumerate}
\item[(i)] $|x_1| \geq \frac{\sigma |a_d| |x|^d}{2}$, \; \; or
\item[(ii)] $|x_1| \leq \eta |y_1|$ and $|y_1| \geq \frac{(1-\sigma) |a_d| |x^d|}{2}.$
\end{enumerate}
\end{lemma}

\begin{proof}
We have $x_1+iy_1=a_d(x+iy)^d+ \sum_{n<d}a_n (x+iy)^n$. Hence if $R'$ is large enough,
$|x_1|+|y_1|\geq |x_1+iy_1| \geq \frac{|a_d| |x+iy|^d}{2}.$ It follows that if Conclusion (i) fails,
then $|y_1|\geq  \frac{(1-\sigma) |a_d| |x^d|}{2}$, but then also it follows
that $|y_1|\geq \frac{1-\sigma}{\sigma} |x_1|$.
\end{proof}

\begin{lemma}\label{lemma2.12}
There exist $R'',\lambda>0$ so that if $x+iy\in \mathbb C$ and $|x|\geq R'',$ then
$|x_1|\geq \lambda |x|^d$ or $|x_2|\geq \lambda^{d+1} |x|^{d^2}.$ In fact in general,
$|x_n|\geq \lambda^{\frac{d^n-1}{d-1}} |x|^{d^n}$ cannot fail for two consecutive positive integers $n.$
\end{lemma}

\begin{proof}
It suffices to prove the first part.
Assume that $|x|\geq R''$ where $R''$ will be chosen sufficiently large.
The first possibility is that $|x|\leq \eta |y|.$ Then $|y|\geq R''/\eta \geq R.$ Hence by Lemma 2.9
$|x_1|\geq \frac{|Ay^d|}{2} \geq \frac{|A| |x|^d}{2\eta^d }$.
The other possibility is that $|x|\geq \eta |y|.$ Then we can conclude Case (i) or (ii) in Lemma 2.10.
In Case (i) we are done, so assume we are in Case (ii). Then the point $(x_1,y_1)$ satisfies the
condition of Lemma \ref{lemma2.10}. Hence
$$
|x_2|\geq \frac{|Ay_1^d|}{2} \geq \frac{1}{2}\cdot |A|  \left(\frac{(1-\sigma) |a_d| |x^d|}{2}\right)^d.
$$
\end{proof}

Theorem \ref{thm:continuity} now follows from the combination of Lemmas \ref{lemma2.10} and \ref{lemma2.12}.

\section{Mirrored Orbits}

The goal of this section is to study which points $z,w$ in $\mathbb{C}$ cannot be distinguished by their real orbits.

\begin{definition}
A point $z\in \mathbb{C}$ is \emph{mirrored} by $w$ if $w \neq z$ and $\mathrm{Re}(z_n) = \mathrm{Re}(w_n)$ for all $n \ge 0$. We say that $z$ is mirrored if there exists a $w$ that mirrors $z$.
\end{definition}

If $P$ has real coefficients then every point $z = x + iy$ with $y \neq 0$ is mirrored by $\bar{z}$.

\begin{definition}
Consider points $z$ and $w$ that mirror eachother. We say that the mirror \emph{breaks} if there exists an $n \in \mathbb{N}$ so that $P^n(z) = P^n(w).$
\end{definition}

Recall that a polynomial $P = \sum_{n=0}^d a_n z^n$ is called exceptional if the equation $\mathrm{Re}(a_d i^d) = 0$ is satisfied. Throughout this section we will assume that the polynomial $P$ is not exceptional.

\begin{lemma}\label{lemma3.3}
Suppose that there exists an open set $U$ of points that are mirrored. Then the mirroring extends to an unbounded connected open set.
\end{lemma}

\begin{proof}
We first write $P(x,y)=(A(x,y),B(x,y))$ as a map on $\mathbb R^2$. Here $A,B$ are real polynomials of two real variables. We complexify this as a map $P'(u,v)=(A'(u,v),B'(u,v))$ from $\mathbb C^2$ to itself, where we replace $x$ and $y$ by independent complex variables.
For $z \in U,$ let $\phi(z)$ be a point that mirrors $z$. We may choose $\phi$ real analytic, and write $V$ for the set of mirror images.
The real analytic map $\phi(x,y)$ between $U$ and $V$ extends to a biholomorphic map $\psi(u,v)=(u,\Lambda(u,v))$ between open sets $U',V'$ in $\mathbb C^2$.  If we write the iterate $(P' )^n (u,v)=(A'_n(u,v),B'_n(u,v))$ then we have, by analytic continuation, the equation $A'_n(u,v)=A_n'(\psi(u,v))$ on $U'.$

By our assumption that $P$ is not exceptional, it follows that the polynomial $A(x,y)$ contains the term $y^d$. Then the level sets of $A'(u,v)$
are branched covers over the $u$ axis. Inside such level sets, the map $\psi(u,v)$ maps points on a level set to the same level set without changing the $u$ coordinate. This map extends by monodromy along any curve in the level set which avoid branch points. The  map might be multiple valued.
We show that for large values of $|u|$ all the solutions of $A'(u,v)=0$ are real values of $v$, when $u$ is real:

Let $L_c=\{A(x,y)=c\}$ for real $c$, $|c|\leq C.$ We want to show that if  $A'(u,v)=A'(u,v')=c$ and $\|(u,v)\|\geq R,$ then $v'$ is real if $u$ is real. This will ensure that the mirror exists on an open set near $\infty$ in $\mathbb C.$ 

For large $\|(u,v)\|$, we have that $|v|\leq K|u|$ if $A'(u,v)=c.$ Moreover, for fixed $u$, the equation
$A'(u,v)=c$ has $d$ complex roots $v$ with multiplicity. However, we see that for $u$ real and large,
$A(u,v)=c$ has already at least $d$ real roots $v:$ We see this by writing $x+iy=re^{i\theta}.$ Then
$P(x,y)=|a_d|r^de^{id\theta+\psi}+\cdots$. The equation Re$(P)=c$ will have at least $d$ solutions.

By analytic continuation this mirroring extends for all $n$ by the same monodromies. We next restrict back to the real coordinates. By the previous observation  we have extended the mirroring to an unbounded connected open set.
\end{proof}

Let us denote by $I_\infty$ the basin of attraction of infinity. Recall that near infinity a polynomial $P = a_d z^d + \cdots + a_0$ is conjugate to $z \rightarrow z^d$. The conjugation map $\phi_\infty$ is called the B\"ottcher map, and if $\phi_\infty(z) = r e^{2\pi i \theta}$ then $r$ and $\theta$ are called the B\"ottcher coordinates of $z$. A curve of the form $\mathrm{Arg}(\phi_\infty(z)) = 2 \pi \theta$ is called an external ray with angle $\theta$. Let $G$ denote the Green function for the filled in Julia set of $P.$

\begin{lemma}
Suppose that $z$ is mirrored by $w$. If $G(w) > G(z)$ then for sufficiently large $n \in \mathbf{N}$ the point $w_n$ must lie on an external ray with angle $\pm \frac{1}{4}$.
\end{lemma}
\begin{proof}
By our assumption $\frac{|w_n|}{|z_n|} \rightarrow \infty$, hence $|\mathrm{Arg}(w_n)|$ must converge to $\frac{\pi}{2}$. As $z \mapsto a_d z^d$ acts expansively on the external rays this convergence can only occur if for all sufficiently large $n \in \mathbb{N}$ the point $w_n$ lies on an external ray with angle $\pm \frac{1}{4}$.
\end{proof}

\begin{corollary}
Except for points on a real one-dimensional set, $z \in I_\infty$ can only be mirrored by points on the same level curve of the Green function.
\end{corollary}

As large sub-level sets of the Green function are strictly convex we obtain the following.

\begin{corollary}\label{cor:double}
For sufficiently large $R > 0$ generic points $z \in \mathbb{C} \setminus D(R)$ can be mirrored by at most one point, which must lie on the same level curve of the Green function.
\end{corollary}

\begin{lemma}\label{lemma:leading}
Suppose that there exists an unbounded open set on which all points are mirrored. Then $a_d$ is real.
\end{lemma}

\begin{proof}
Suppose a generic $z$, with $|z|$ large, is mirrored. As noted above, $z$ can only be mirrored by the point $w$ lying on the other intersection of the level curve of $G$ with the vertical line through $z$. In particular we notice that $\mathrm{Arg}(w) + \mathrm{Arg}(z) \rightarrow 0$ as $|z| \rightarrow \infty$. The same must hold for $P(w)$ and $P(z)$, unless $P(z) = P(w)$ which will not hold on an open set.

Denote the argument of $a_d$ by $\theta$. As $|z| \rightarrow \infty$ the arguments of $P(w)$ and $P(z)$ are approximately equal to the arguments of $a_d w^d$ and $a_d z^d$, which are $d \mathrm{Arg}(w) + \theta$ and $d \mathrm{Arg}(z) + \theta$. Hence we see that
$$
d \mathrm{Arg}(w) + \theta \sim - (d \cdot (- \mathrm{Arg}(w)) + \theta),
$$
which means that the argument $2 \theta$ must be close to $0$. As the error goes to $0$ as $|z| \rightarrow \infty$ while $\theta$ is fixed, it follows that $a_d$ must be real.
\end{proof}

We will proceed to prove the following.

\begin{theorem}\label{thm:mirrors}
Suppose that there exists a non-empty open set $U$ so that every $z \in U$ is mirrored. Then $P$ has real coefficients.
\end{theorem}
\begin{proof} 
By our assumption on $U$ and by Lemmas \ref{lemma3.3} and \ref{lemma:leading} the leading coefficient of $P$ is real. Hence we may conjugate with a linear map $z \mapsto az + b$, with $a \neq 0$ real, such that $P$ becomes of the form $z \mapsto z^d + a_{d-2} z^{d-2} + \cdots + a_0$. Such a conjugation maps mirrored points to mirrored points.

As noted above, the open set $U$ extends to an unbounded open set on which mirroring occurs. Hence we may well assume that $U$ lies in the neighborhood of infinity where the B\"ottcher coordinates are defined. Note that points in $P^n(U)$ must be mirrored for all $n \in \mathbb{N}$, unless mirrors break which can only happen on $1$-dimensional subsets. Let $\phi_\infty(z) = z + l.o.t.$ be holomorphic in a neighborhood of infinity such that $\phi_\infty(P(z)) = \phi_\infty(z)^d$ for $z$ sufficiently large. It follows that for $N$ sufficiently large the set $\phi(P^N(U))$ contains an annulus centered at the origin.

For the purpose of a contradiction we assume that the coefficients of $P$ are not all real, and write $P = g + ih$, where $g$ and $h$ have real coefficients. We write
\begin{equation*}
g(z) = z^d + a_{d-2}z^{d-2} + \cdots + a_0, \; \; h(z) = b_k z^k + \cdots + b_0,
\end{equation*}
where $b_k \neq 0$, and without loss of generality we may assume that $b_k>0$.

We consider the image under $P$ and $P^2$ of an interval $I = \{R + iy\}$, with $R>0$ large and $|y| < 1$. Note that
\begin{equation*}
g(R+iy) = g(R) + i d R^{d-1}y - \binom{d}{2}R^{d-2}y^2 + O(R^{d-4})O(y^2)+iO(R^{d-3})O(y),
\end{equation*}
and
\begin{equation*}
i h(R + iy) = ih(R) - kb_k R^{k-1}y + iO(R^{k-2})O(y^2)+O(R^{k-2})O(y).
\end{equation*}
Hence 
$$
\begin{aligned}
X(y):= &{\mbox{Re}}(P(R+iy)) \\
= g(R)- \binom{d}{2}R^{d-2}y^2 + & \mathcal O(R^{d-4})O(y^2)
-kb_kR^{k-1} y+\mathcal O(R^{k-2})O(y).
\end{aligned}
$$

It follows that $X_{yy}<0$ on the interval $I$  for $R>0$ sufficiently large.

Choose $C>0$ sufficiently large so that for $R>0$ large enough and $y_1 = C \cdot R^{k+1-d}$ we have
\begin{equation*}
\binom{d}{2}R^{d-2}|y_1|^2 >> k |b_k| R^{k-1} |y_1|.
\end{equation*}

Then we have that $X_y(y_1) < 0 < X_y(- y_1)$ and that $X(\pm 1)<X(y_1)$.
For $R>0$ sufficiently large the vertical line through $P(R+iy_1)$ therefore intersects $P(I)$ in exactly one other point, say $P(R + iy_2)$.

Note that $g(R)>X(y_1)$ and hence  $y_2<0.$ One gets approximately
$$
y_2=-y_1-\delta b_k R^{k+1-d}, \; \; \textrm{with} \; \; 0<\delta<< C.
$$

Solving
\begin{equation*}
\mathrm{Re}(P(R+iy_2)) = \mathrm{Re}(P(R+iy_1))
\end{equation*}
one obtains
\begin{equation*}
\mathrm{Im}(P(R + iy_2)) \sim R^k (b_k + \tilde{C}), \; \; \mathrm{and} \; \;  \mathrm{Im}(P(R + iy_1)) \sim R^k (b_k - \tilde{C}).
\end{equation*}
But then for $R$ large enough one has $\mathrm{Re}(P^{\circ 2}(R+iy_2)) > \mathrm{Re}(P^{\circ 2}(R+iy_1))$. Hence for $R$ sufficiently large and generic $y$ comparable to $C \cdot R^{k+1-d}$, the point $R+iy$ cannot be mirrored. However, since $\phi_\infty(P^N(U))$ contains a large annulus centered at the origin, it follows that $P^N(U)$ must contain such points $R+iy$, which contradicts our assumption that all points in $U$ are mirrored.
\end{proof}

\begin{lemma}\label{counting}
Let $X=\{(z,w) \mid z \; \mathrm{mirrors} \; w \}$.
Suppose that $z$ lies on a vertical line which is not invariant. Then there are at most $d^2-1$ points
$w$ which mirror $z$.
\end{lemma}

\begin{proof}
Let $x_0,x_1,x_2$ be the first three points in the orbit. Let $L_0,L_1,L_2$ denote the vertical lines
through these points. If $P(L_0)\neq L_1,$ then by Lemma 2.1, there are at most $d$ points in $L_0$ mapped to $L_1.$ So the mirror can have at most $d-1$ points. If $P(L_0)=L_1$, then by our hypothesis that $P$ is not exceptional we have
$L_1\neq L_0$ and hence by Lemma 2.1, $P(L_1)\neq L_2.$ But then at most $d$ points in $L_1$ can be mapped to $L_2$  and each of those have at most $d$ primages in $L_0.$ So there are at most
$d^2-1$ points in the mirror of $z.$
\end{proof}

\begin{lemma}
The dimension of $X$ is at most $2.$
\end{lemma}

\begin{proof}
If we include the diagonal, the dimension is at least $2.$ If there is an invariant vertical line $L$ then
$L\times L$ consists of mirrored points. This set is two dimensional. Except for this, there is for every point $z$ at most a finite number of points $w$ for which $(z,w)$ are mirrored. This excludes components of dimension $3$ or more.
\end{proof}

We suppose next that $X$ contains a $2$ dimensional component, distinct from the diagonal and $L\times L$. Then there is  a small piece of this which is an unbranched cover over an open set
$U$ in the $z$ axis. Then by Theorem 3.8, $P$ has real coefficients. Then $X$ is a graph over $U$, with $w=\phi(z)$. Necessarily
$\phi=x+i\lambda(x,y)$ and $\lambda$ is real analytic. We can shrink $U$ and make $\phi$ a diffeomorphism between $U$
and $V=\phi(U)$: If the mirroring is not conjugation, then by continuation to $\infty$ we have at least
three mirrored points which contradicts Corollary 3.6.

In conclusion we have that :

\begin{lemma}
The set of mirrored points is of dimension at most 1, with the exception of conjugate points in the case where $P$ has real coefficients.
\end{lemma}

An easy way to obtain a curve of mirrored points is the following. Start with a generic $z$ with two pre-images $z_1,z_2$ having the same real part.
Then there is a biholomorphic map $\phi:U(z_1)\rightarrow U(z_2)$  so that $w$ and $\phi(w)$
have the same image. In a neighborhood of $z_1$, there is a zero set of the function
$(Re)(\phi(w)-w)$ which is  a real curve. This is a curve of mirrored points.


\section{Topological Entropy}

Let $X$ be a metrizable compact topological space and let $F:X \rightarrow X$ be a continuous map.
If $U$ is an open neighborhood of the diagonal $\Delta:=\{(x,x);x\in X\}$ we define $(U,n)$ balls
$B(x,U,n)$, centered at $x\in X$ by
$$
B(x,U,n)=\{y\in X; \{(x,y),(F(x),F(y)),\dots (F^{n-1}(x),F^{n-1}(y))\} \subset U\}.
$$

We can use these balls to define metric and topological entropy. In this section we discuss topological entropy. Metric entropy is introduced in the next section.

We set $N(U,n)$ to be the maximum of the number of pairwise disjoint $B(x,U,n)$ balls one can have in $X.$ Next we define
$$
H_{U,X}=H_U:=\limsup \frac{1}{n} \log N(U,n).
$$
If $V\subset U$ are two neighborhoods of the diagonal, then $H_V\geq H_U.$ Let $U_1\supset U_2\cdots \supset U_n\cdots$ be any neighborhood basis of the diagonal. Then we can define the topological entropy
$$
h_{\mbox{top}}(F,X)= \lim_{n\rightarrow \infty}H_{U_n}.
$$

If we let $\rho$ denote any metric defining the topology of $X,$ then we can set
$U_\epsilon:= \{(x,y): \rho(x,y) < \epsilon\}$. We call the balls $B(x,\epsilon, n)$ and
$N(U_\epsilon,n)=:N(\epsilon,n)$, $H_{U_\epsilon}=:H_\epsilon.$ So then $h_{\mbox{top}}=
\lim_{\epsilon\rightarrow 0} H_\epsilon.$

We will investigate how entropy behaves under semi-conjugacies. Let $F:X\rightarrow X$ and $G:Y\rightarrow Y$ be semi-conjugate continuous maps on the spaces $X$, and $Y$, i.e. there exists
a continuous map $\Phi:X\rightarrow Y$ so that $G\circ \Phi=\Phi\circ F.$

\begin{lemma}
If $\Phi$ is surjective, then $h_{\mbox{top}}(F,X)\geq h_{\mbox{top}}(G,Y).$
\end{lemma}

\begin{proof}
Let $U$ be a neighborhood of the diagonal in $Y^2.$ Choose $n$ and pick a family of $N(U,n)$ disjoint
balls $B(x_j,U,n)$ in $Y.$ Let 
$$
V=\{(u,v)\in X^2;(\Phi(u),\Phi(v))\in U\}.
$$
Then $V$ is a neighborhood of
the diagonal in $X^2.$ Since $\Phi$ is surjective, we can find $y_j\in X$ so that $\Phi(y_j)=x_j.$
We will show that the balls $B(y_j,V,n)$ are pairwise disjoint. It suffices to prove that if $z\in B(y_j,V,n)$ then $\Phi(z)\in B(x_j,U,n).$
We have that for all $0\leq \ell \leq n-1$, $F^\ell(y_j,z)\in V.$ Hence $(\Phi(F^\ell(y_j)),\Phi(F^\ell (z)))\in U.$
Therefore, $(G^\ell(\Phi(y_j)),G^\ell(\Phi(z)))\in U.$ This implies that 
$$
\Phi(z)\in B(x_j,U,n).
$$
It follows that $H_{V,X}\geq H_{U,Y}.$ Therefore $h_{\mbox{top}}(F,X)\geq H_{U,Y}$
for all neighborhoods $U.$ From this the lemma follows.
\end{proof}

Recall that for a non-exceptional polynomial $P$ it was shown that the map $\Phi : \mathbb P^1 \rightarrow  \overline{S}$ is surjective. We conclude the following.

\begin{corollary}
If the non-exceptional polynomial $P$ has degree $d$ then the topological entropy of the map $Q$ is at most $\log(d)$.
\end{corollary}

Recall from Lemma \ref{counting} that if $P$ is a non-exceptional polynomial, then for every $z \in \mathbb P^1$ there are at most $M = d^2 -1$ points $w$ with $\Phi(w) = \Phi(z)$.

We denote by $d(\cdot, \cdot)$ the spherical metric on $\mathbb P^1$. For $k \in \mathbb N$ we define
$$
U_k := \{(z,w) \mid d(z,w) < \frac{1}{k} \}.
$$
Note that the sets $U_k$ form a neighborhood basis for the diagonal in $\mathbb P^1 \times \mathbb P^1$.

Let $\delta = \delta(k)>0$ be sufficiently small such that if $z_0, w_0 \in \mathbb P^1$ are such that $d(z_0,w_0) < 2 \delta$, then $d(z_j, w_j) < \frac{1}{k}$ for $0 \le j \le k-1$. Let $\overline{x} \in \overline{S}$ and write $\Phi^{-1}(\overline{x}) = \{a_j\}_{j=1}^r$. By continuity of the map $\Phi$ and compactness of $\mathbb{P}^1$ there exists a neighborhood $\Omega(\overline{x})$ whose lift to $\mathbb P^1$ satisfies
$$
\Phi^{-1}(\Omega) \subset \bigcup_{j=1}^r D(a_j, \frac{\delta}{2}).
$$
Then let $\eta = \eta(\overline{x}) >0$ be such that
$$
\bigcup_{j=1}^r D(a_j, \eta) \subset \Phi^{-1}(\Omega),
$$
and define $\mathcal{N}(\overline{x})$ to be the maximal subset of $\bigcup_{j=1}^r D(a_j, \frac{\eta}{3})$ for which
$$
\Phi^{-1} \circ \Phi (\mathcal{N}(\overline{x})) = \mathcal{N}(\overline{x}).
$$
Note that $\mathcal{N}(\overline{x})$ is obtained from $\bigcup_{j=1}^r D(a_j, \frac{\eta}{3})$ by removing the relatively closed set of points whose mirrors do not all lie in $\bigcup_{j=1}^r D(z_j, \eta)$. We then define
$$
V_k := \bigcup_{\overline{x} \in S} \mathcal{N}(\overline{x}) \times \mathcal{N}(\overline{x}).
$$

\begin{lemma}
The sets $\Phi(V_k)$ form a neighborhood basis of the diagonal in $\overline{S}\times \overline{S}$, and $\Phi^{-1} \Phi(V_k) = V_k$ for each $k \in \mathbb N$.
\end{lemma}

Let $z \in \mathbb P^1$ and consider all points $\overline{x} \in \overline{S}$ for which $z \in \mathcal{N}(\overline{x})$. Let $\mu$ be the supremum over all $\eta(\overline{x})$. Then there certainly exists an element in $S$ with $z \in \mathcal{N}(\overline{x})$ for which $\eta(\overline{x}) \ge \frac{2}{3} \mu$. From now on let $\overline{x}$ be such an element.

Let $\overline{y}$ be any other element of $\overline{S}$ for which $z \in \mathcal{N}(\overline{y})$. Let us write $\Phi^{-1}(\overline{x}) = \{a_j\}_{j=1}^r$ and $\Phi^{-1}(\overline{y}) = \{b_j\}_{j=1}^s$.

\begin{lemma}
$$
\mathcal{N}(\overline{y}) \subset \bigcup_{j=1}^r D(a_j, \delta).
$$
\end{lemma}
\begin{proof}
Since $z\in \mathcal{N}(\overline{y})$ we may assume that $d(z, b_1) < \frac{\eta(\overline{y})}{3}$. Similarly we have $d(z, a_1) < \frac{\eta(\overline{x})}{3}$. Since $\eta(\overline{y}) \le \mu$ and $\eta(\overline{x}) \ge \frac{2}{3} \mu$ it follows that
$$
d(a_1, b_1) < \frac{\eta(\overline{x})}{3} + \frac{\eta(\overline{y})}{3} \le (\frac{1}{3} + \frac{1}{2}) \eta(\overline{x}) < \eta(\overline{x}).
$$
Hence $b_1 \in \Phi^{-1}(\Omega(\overline{x}))$, and therefore $\Phi^{-1}(\overline{y}) \subset \Phi^{-1}(\Omega(\overline{x}))$. Since
$$
\begin{aligned}
\Phi^{-1}(\Omega(\overline{x})) & \subset \bigcup D(a_j, \frac{\delta}{2}), \; \; \mathrm{and}\\
\mathcal{N}(\overline{y}) & \subset \bigcup D(b_j, \frac{\delta}{2}),
\end{aligned}
$$
it follows that
$$
\mathcal{N}(\overline{y}) \subset \bigcup D(a_j, \delta).
$$
\end{proof}

Let $z \in \mathbb P^1$ and define
$$
W_z := \{w \in \mathbb P^1 \mid (z,w) \in V_K\}.
$$
Recall from Lemma \ref{counting} that for each $\overline{x} \in \overline{S}$ the set $\Phi^{-1}(\overline{x})$ contains at most $M = d^2-1$ elements.

\begin{corollary}\label{cor:disks}
The set $W_z$ is contained in at most $M$ disks of radius $\delta$. Moreover, there exists an $\overline{x} \in \overline{S}$ so that the centers of these disks can be chosen to lie in the set $\Phi^{-1}(\overline{x})$.
\end{corollary}

In order to estimate the topological entropy of $Q$ from below we will work with maximally separated sets with respect to the neighborhood bases $\{U_k\}$ and $\{(\Phi\times \Phi)(V_k)\}$. Notice that a collection of points $X \subset \mathbb P^1$ is $(n,V_k)$-separated if and only if $\Phi(X) \subset S$ is $(n, (\Phi\times \Phi)(V_k))$-separated. This is useful as it allows us to only work in $\mathbb P^1$.

\begin{theorem}
Suppose that $P$ is a non-exceptional polynomial. Then the topological entropy of $Q$ is $\mathrm{log}(d)$.
\end{theorem}
\begin{proof}
Let $k \in \mathbb N$, and define $U_k, V_k \subset \mathbb P^1 \times \mathbb P^1$ as above. Let $n \in \mathbb N$ and suppose that $X$ is a collection of points in $\mathbb{P}^1$ that are $(n, U_k)$-separated. We would like to estimate from below the minimal number of points in $X$ whose images under $\Phi$ are $(n, (\Phi\times \Phi)(V_k))$-separated, or equivalently, the minimal number of points in $X$ that are $(n, V_k)$-separated. Let $z \in X$, and let $Y \subset X$ contain the points $w^i$ that are not $(n, V_k)$-separated from $z$. We will estimate the size of the set $Y$ from above.

By Corollary \ref{cor:disks} the only possible way for $(z_j, w^i_j)$ to lie in $V_k$ but not in $U_k$ is for the pair to lie in two distinct disks of radius $\delta = \delta(k)$ that contain the set $W_z$. Moreover, the centers of these disks must then be at least $\frac{1}{k} - 2\delta$ apart. If $z_j$ and $w^i_j$ instead lie in the same disk of radius $\delta$, then by our assumption on $\delta$ it follows that $(z_l, w^i_l)$ will lie in $U_k$ for $l = j, \ldots , j+k-1$.

It follows that we can represent the points in $Y$ by unique words in the letters $1$ through $M$. Let us be more precise. At time $0$ our set $Y$ is covered by at most $M$ disks. We assign to these disks the letters $1$ through (at most) $M$. The first letter we assign to each element $w^i \in Y$ is naturally the letter of a disk that contains the point $w^i_0 = w^i$.

Then consider the first time for which two points in $Y$ with the same initial letter have drifted at least $\frac{1}{k}$ apart, say after $j$ iterates. At that time we again assign the (at most) $M$ disks a letter, and give each element $w^i$ its second letter, namely the letter corresponding to a disk that contains $w^i_j$. Later letters are assigned similarly. By our assumption on $\delta(k)$ each word has at most $\frac{n}{k}+1$ letters, and by our assumption that the points in $Y$ are $(n, U_k)$-separated it follows that the words corresponding to these points are all unique. We can therefore estimate the number of elements in $Y$ by
$$
|Y| \le M^{\frac{n}{k}+1}.
$$
We denote by $N(n,U_k)$ and $N(n,V_k)$ the maximal number of respectively $(n, U_k)$- and $(n, V_k)$-separated points in $\mathbb P^1$. It follows that if $X$ contains $N(n,U_k)$ points, then at least
$$
N(n,U_k)/ M^{\frac{n}{k}+1}
$$
of those points are $(n,V_k)$-separated. In other words,
$$
N(n,U_k) \le M^{\frac{n}{k}+1} \cdot N(n,V_k).
$$
It follows that
$$
\frac{1}{n} \log N(n,U_k) \le (\frac{1}{k}+\frac{1}{n}) \log M + \frac{1}{n} \log N(n,V_k).
$$
Hence the topological entropy of $Q$ is at least $\log(d) - \frac{1}{k} \log M$, which holds for all $k \in \mathbb N$. Hence $h_{top}(Q) \ge \log(d)$. As we already estimated the entropy from above the proof is complete.
\end{proof}


\section{Metric Entropy}

We recall the notation from the previous section. Let $X$ be a metrizable compact topological space, and let $F:X \rightarrow X$ be a finite continuous map.
If $U$ is an open neighborhood of the diagonal $\Delta:=\{(x,x);x\in X\}$ we define $(U,n)$ balls
$B(x,U,n)$, centered at $x\in X$ by
$$
B(x,U,n)=\{y\in X; \{(x,y),(F(x),F(y)),\dots (F^{n-1}(x),F^{n-1}(y))\} \subset U\}.
$$

Let $\lambda$ be a probability measure on $X.$
$$
h_{\lambda}(F,x,U)=  \liminf_n -\frac{1}{n} \log (\lambda(B(x,U,n)))
$$
We notice that $h_{\lambda}(F,x,U)$ increases when we replace $U$ by a smaller set.
$$
h_{\lambda}(F,x)= \sup_{U} h_{\lambda}(F,x,U)
$$
If $\lambda$ is backwards invariant, then $h_\lambda(F,x)\geq h_\lambda(F,F(x))$. If $\lambda$ is also ergodic, then this function is constant $d\lambda$-almost everywhere. The metric entropy $h_{\lambda}(F)$ is defined to be this constant.

Let $\rho$ be any metric on $X$ defining the topology of $X.$ We can then use the neighborhoods
of the diagonal $U_\epsilon:=\{(x,y); \rho(x,y)<\epsilon\}$. We call the balls $B(x,\epsilon,n).$ 
Also we set $h_{\lambda}(F,x,\epsilon)=h_{\lambda}(F,x,U_\epsilon)$ and
then 
$$
h_{\lambda}(F,x)= \sup_{\epsilon} h_{\lambda}(F,x,\epsilon).
$$
Clearly, the metric entropy is independent of the metric. In fact it is a toplogical invariant.

We assume next that $P$ is a non-exceptional polynomial on $\mathbb C.$
Let $\lambda$ be a probability measure on $\mathbb P^1$, and let $\nu=\Phi_* \lambda$
be the push forward to $\overline{S}$. 

\begin{lemma}
If $\lambda$ is invariant on $\mathbb P^1$, then the push-forward $\nu$ is invariant on $\overline{S}.$ If $\lambda$ also is ergodic, then the push-forward, $\nu,$ is ergodic as well.
\end{lemma}
\begin{proof}
We prove first 
that invarance of $\nu$. Let $E$ be a Borel set in $\overline{S}$.
$$
\begin{aligned}
\nu(E) & =  \lambda(\Phi^{-1}(E))\\
& = \lambda(P^{-1}(\Phi^{-1}(E))\\
& = \lambda((\Phi\circ P)^{-1}(E))\\
& =  \lambda((Q\circ \Phi)^{-1}(E))\\
& =  \lambda((\Phi^{-1} (Q^{-1}(E))\\
& =  \nu(Q^{-1}(E))
\end{aligned}
$$

Next we prove ergodicity. Let $E\subset \overline{S}$ be a Borel set. Assume that
the set $Q^{-1}(E)$ is the same as $E$  except for a set of $\nu$-measure $0.$
So the sets $Q^{-1}(E)\setminus E$ and $E\setminus Q^{-1}(E)$ have $\nu$-measure $0.$
The pull backs of $E$, $Q^{-1}(E)$ and the two difference sets to $\mathbb P^1$ have the same
$\lambda$-measures. Consider the sets $F=\Phi^{-1}(E)$ and $P^{-1}(F).$
Since $P^{-1}(\Phi^{-1}(E))=\Phi^{-1}(Q^{-1}(E))$ we see that $F$ is invariant modulo sets of measure $0$ for $P.$ Hence $F$ has measure $0$ or $1$. Hence the $\nu$-measure of $E$ is also $0$ or $1.$
\end{proof}

\begin{lemma}
$h_\lambda(P,z) \geq h_\nu(Q,\Phi(z))$.
\end{lemma}

\begin{proof}
Choose a neighborhood $V$ of the diagonal in $\overline{S}\times \overline{S}.$ Then
$$
U:=\{(z,w)\in \mathbb P^1 \times \mathbb P^1; (\Phi(z),\Phi(w))\in V\},
$$
is a neighborhood of the diagonal in $\mathbb P^1 \times \mathbb P^1.$ If $w \in B(x,U,n),$ then
$$
\{(z,w),\cdots,(P^{n-1}(z),P^{n-1}(w))\}\subset U.
$$
Hence
$$
\{(\Phi(z),\Phi(w)),\cdots,(\Phi(P^{n-1}(z)),\Phi(P^{n-1}(w)))\}\subset V.
$$
This implies that
$$
\{(\Phi(z),\Phi(w)),\cdots,(Q^{n-1}(\Phi(z)),Q^{n-1}(\Phi(w)))\}\subset V.
$$
Therefore $\Phi(w) \in B(\Phi(z),V,n)$, and so we get:
$$
\begin{aligned}
\Phi(B(z,U,n)) & \subset B(\Phi(z),V,n)\\
& \Rightarrow & \\
\lambda(B(z,U,n)) & \leq \nu(B(\Phi(z),V,n))\\
& \Rightarrow & \\
h_\lambda (P,z,U) & \geq h_\nu (Q,\Phi(z),V)\\
& \Rightarrow & \\
h_\lambda (P,z) & \geq h_\nu (Q,\Phi(z),V)\\
& \Rightarrow & \\
h_\lambda (P,z) & \geq h_\nu (Q,\Phi(z)).
\end{aligned}
$$
\end{proof}

\begin{lemma} Assume that $\nu$ is invariant and ergodic. Then for $d\lambda$-almost every $z$, we have 
$h_\lambda(P,z) \geq h_\nu(Q)$
\end{lemma}

\begin{proof}
This follows since $h_\nu(Q,y)=h_\nu(Q)$ almost everywhere $d\nu.$
Hence the inequality also holds $d\lambda$-almost everywhere, as $\nu$ is the
pushforward of $\lambda.$
\end{proof}

\begin{corollary} Assume that $\lambda,\nu$ are invariant and ergodic. Then 
$h_\lambda(P) \geq h_\nu(Q).$
\end{corollary}

The next lemma applies to a space $X$ with a selfmap $F$.

\begin{lemma}
Suppose that $\sigma$ and $\tau$ are two invariant, nonzero, positive, measures and let $\eta:=\sigma+\tau$.
Then $\eta$ is invariant. If $\eta$ is ergodic, then both $\sigma$ and $\tau$ are ergodic.
Moreover, they must both be multiples of $\eta.$
\end{lemma}
\begin{proof}
Invariance of $\eta$ is clear. \\
Next, assume that $\eta$ is ergodic. We can suppose that $\eta$ is a probability measure. Suppose that $E$ is a set with positive $\sigma$-measure. 
Then $\tau(E)>0:$ Let $H:=\cap_{k} \cup_{n\geq k} F^{-n}(E).$ By invariance of $\tau$ this set has positive $\tau$-measure, hence also positive $\eta$-measure. The set is also invariant, hence
$\eta(H)=1.$ Therefore also $\tau(H)>0.$ But this implies that also $\tau(E)>0.$

Suppose next that $\sigma$ is not ergodic. Then there exists a set $E$ with
$0<\sigma(E)<\sigma(X)$ which is invariant, i.e. the sets $F^{-1}(E)\setminus E$ and
$E\setminus F^{-1}(E)$ both have $\sigma$-measure $0.$ This implies that the same is
true for $\tau.$ Hence $E$ is also invariant for $\tau$ and therefore also for $\eta.$
But $0<\eta(E)<\eta(X)$, contradicting ergodicity of $\eta.$\\
It remains to show the last statement.
Let $\sigma'=\sigma/\sigma(X), \tau'=\tau/\tau(X), \eta'=\eta/\eta(X).$ So these are
invariant ergodic probability measures. Let $E$ be a set. Then $\chi\circ F^n$ converge to a constant function, $\sigma'(E)=\tau'(E)=\eta'(E).$ Proportionality follows.
\end{proof}

\begin{lemma}
If $\nu$ is invariant ergodic on $\overline{S}$, then $\nu$ is the push forward of an invariant ergodic measure on $\mathbb P^1.$
\end{lemma}
\begin{proof}
We define inductively a sequence  $(\lambda_n)$ of measures on $\mathbb P^1$, all of which
have $\nu$ as push forward. We would like to find a measure which is also backwards invariant under $P$.

We divide $\mathbb P^1$ into finitely many sets $A_j$ where $A_j$ consists of those points for which the fibers of $\Phi$ have exactly $j$ points. So the mirrors consist of the points in $\cup_{j>1} A_j.$
We define $\lambda_1$ by dividing the measure of $\Phi(A_j)$ into $j$ equal measures. 

We inductively define $\lambda_{n+1}(E)=\lambda_n(P^{-1}(E)).$ Next we show that this procedure
keeps the property of the pushforward being $\nu.$ Let $F$ be a subset of $\overline{S}$ and let $E$ denote $\Phi^{-1}(F).$ Then 
$$
\begin{aligned}
\lambda_{n+1}(E) & = \lambda_n(P^{-1}(E))\\
& = \lambda_n(P^{-1}(\Phi^{-1}(F)))\\
& = \lambda_n(\Phi^{-1}(Q^{-1}(F)))\\
& = \nu(Q^{-1}(F))=\nu(F).
\end{aligned}
$$

Next, we use Cesaro means. Set $\lambda'_n=\frac{1}{n} \sum_{1\leq j\leq n}\lambda_j.$
The pushforward of these measures are all equal to $\nu.$ 
The total mass of the signed measure $\lambda'_{n+1}-\lambda'_n$ is at most
$2/(n+1).$ 
We get 
$$
\begin{aligned}
|\lambda'_{n+1}(E)-\lambda'_{n+1}(P^{-1}(E))| & \leq 
|\lambda'_{n+1}(E)-\lambda'_{n}(P^{-1}(E))|+\frac{2}{n+1}\\
& =
|\lambda'_{n+1}(E)-\frac{1}{n} \sum_{1\leq j\leq n}\lambda_j(P^{-1}(E))|+\frac{2}{n+1}\\
& = |\lambda'_{n+1}(E)-\frac{1}{n} \sum_{2\leq j\leq n+1}\lambda_j(E)|+\frac{2}{n+1},
\end{aligned}
$$
and therefore
$$
\begin{aligned}
|\lambda'_{n+1}(E)-\lambda'_{n+1}(P^{-1}(E))| & \leq |\lambda'_{n+1}(E)-\frac{1}{n+1} \sum_{1\leq j\leq n+1}\lambda_j(E)|+\frac{4}{n+1}\\
& = \frac{4}{n+1}.
\end{aligned}
$$
It follows that any weak limit is invariant on $\mathbb P^1$, and also that $\nu$ is the
pushforward of any weak limit.
Pick such a limit $\eta$. Suppose that $\eta$ is not ergodic. Then
we can write $\eta=\sigma+\tau$ where both measures are nonzero and invariant and their supports are disjoint. Let $\sigma',\tau'$ be their pushforwards. These measures are invariant by Lemma 5.1
and their sum is ergodic. Hence by Lemma 5.5, they are both ergodic
and in fact $\sigma'=a\nu$ and $\tau'=b\nu$ for $a+b=1$ and $0<a,b.$
Hence we have that $\nu$ is the pushforward of two measures, with disjoint supports.
This implies that $\nu$ cannot charge the complement of the mirrors. We can repeat the argument
if neither $\sigma$ nor $\tau$ is ergodic. We then see that $\nu$ cannot charge any mirror with
at most three points in the mirror. Repeating the procedure finitely many times and using that
there is an upper bound on the number of mirrored points, we see that
one of the measures obtained must be ergodic.
Hence $\nu$ must be the push forward of an invariant, ergodic measure.
\end{proof}

Now we denote by $\mu$ the measure of maximal entropy on $\mathbb P^1$. We have the following.

\begin{lemma}
The metric entropy of the push-forward of the measure $\mu$ is the same as the metric entropy of the
measure $\mu$
\end{lemma}

\begin{proof}
We choose a point $\overline{x}_0$ in $\overline{S}$ and want to estimate the entropy function there.
For $j = 1 \ldots m_0$ let $z_0^j$ denote the corresponding mirrored points in case there is a mirror there.
If not, we consider only $z_0^1$. Suppose that the set $\{z_n^1,\dots ,z_n^j\}$ has $m_j$ distinct points. Then we know that $m_0\geq m_1\geq \cdots \ge 1.$ So the number of points decreases until it eventually stabilizes. 

We now assume that we have $m$ mirrored points $z_0^1,\dots,z_0^m$ and that
for all $n,$ the points $z_n^1,\dots,z_n^m$ are disjoint. 

Choose a metric $\rho$ on $\overline{S}.$ Let $\epsilon>0.$ We will estimate the $\nu$ measure
of the balls $B(\overline{x}_0,n).$ Since the measure $\nu$ is supported on the image of the Julia set, we can assume that at least one $z_0^j$ is in $J.$ We assume that  $z_0^1, \dots, z_0^s$ are the points in the mirror with orbits not converging to a periodic critical orbit. (Such points are not in $J.$)
If $\epsilon>0$ is small enough, then the part of $\Phi^{-1}(B(\overline{x}_0,\epsilon,n))$ which is near
$z_0^{s+1},\dots,z_0^m$ carries no mass.

Let $k$ be a large integer. We choose a $\delta_0>0$ small enough so that the fraction of integers $n$ for which some $z_n^j, j=1,\dots,s$ is closer than $2\delta_0$ to a critical point is at most $1/k.$ 

For $\ell<n$ we define the $\epsilon$-balls in $\mathbb P^1$, 
$$
B'(n,\ell,\epsilon) := \bigcap_{r = 0}^{n-\ell}  \{w \in \mathbb P^1 \; \mid \; \rho(\Phi(P^r(w)), \overline{x}_{\ell+r}) < \epsilon \}.
$$
We see that as long as the points $z_{n-1}^1, \ldots , z_{n-1}^s$ have distance at least $2\delta_0$ to the critical points and $\epsilon$ is small enough that $P$ is one to one on $B(z_{\ell-1}$, then 
$$
\mu (B'(n,\ell-1,\epsilon)) \le \mu(B'(n,\ell,\epsilon))/d.
$$
On the other hand we always have that
$$
\mu (B'(n,\ell-1,\epsilon)) \le \mu(B'(n,\ell,\epsilon)).
$$
It follows that the measure of $B'(n,0,\epsilon)$ is at most $C\left(\frac{1}{d}\right)^{n-n/k}$. Therefore the metric entropy on $\overline{S}$ is at least $\log d.$
\end{proof}


\section{Proof of Theorem 1.4}

In this section we restate and prove the main theorem of the paper.

\medskip

\noindent {\bf Theorem 1.4.} \emph{Let $P$ be a non-exceptional complex polynomial of degree $d \geq 2$.
Then the probability measure $\nu$ is invariant and ergodic. Moreover it is the unique measure of maximal entropy, $\log d.$}

\medskip

\begin{proof}
It follows from Lemma 5.1 that $\nu$ is invariant and ergodic. From Lemma 5.7 it follows that the metric entropy of $\nu$ is $\log d.$ From Theorem 4.6 we see that $P$ also has topological entropy
$\log d.$ If $\sigma\neq \nu$ is any other invariant ergodic probability measure on $S$, then 
by Lemma 5.6, $\sigma$ is the pushforward of an invariant ergodic probability measure
$\tau$ on $\mathbb P^1.$ Necessarily, $\tau \neq \mu.$ Hence  the metric entropy of $\tau$
is strictly less than $\log d.$ It follows by Corollary 5.4 that the metric entropy of $\sigma$ is strictly less than $\log d.$ Therefore $\nu$ is the unique measure of maximal entropy.
\end{proof}

\end{document}